\documentclass[12pt,reqno]{amsart}
\usepackage{amsmath,amssymb,tikz-cd,stmaryrd,tikz-cd,enumitem,mathtools,microtype,hyperref,amsthm,array,booktabs,float}
\usepackage{cleveref}
\crefname{subsection}{subsection}{subsections}

\usepackage[hyphenbreaks]{breakurl}

\theoremstyle{definition}
\newtheorem{theorem}{Theorem}[section]
\newtheorem*{theorem*}{Theorem}
\newtheorem{lemma}[theorem]{Lemma}

\newtheorem{example}[theorem]{Example}
\newtheorem{definition}[theorem]{Definition}
\newtheorem{remark}[theorem]{Remark}

\newtheorem*{question*}{Question}
\newtheorem*{remark*}{Remark}

\DeclareMathOperator{\GL}{GL}

\DeclareMathOperator{\BesselK}{BesselK}
\DeclareMathOperator{\BesselI}{BesselI}

\newcommand{\cs}{\mathbin{\circledS}}
\newcommand{\pda}{\mathord{\downarrow}}

\DeclareMathOperator{\GCRD}{GCRD}
\DeclareMathOperator{\ord}{ord}

\DeclareMathOperator{\Minop}{MinOp}

\DeclareMathOperator{\RightFactors}{RightFactors}

\DeclareMathOperator{\Sym}{Sym}

\def\namedlabel#1#2{\begingroup
    #2%
    \def\@currentlabel{#2}%
    \phantomsection\label{#1}\endgroup
}

\definecolor{c1}{RGB}{203, 75, 48}

\title{Algorithms for 2-Solvable Difference Equations}

\makeatletter
\newcommand\footnoteref[1]{\protected@xdef\@thefnmark{\ref{#1}}\@footnotemark}
\makeatother

\providecommand{\keywords}[1]
{
  \small	
  \textbf{\textit{Keywords---}} #1
}

\author{Heba Bou KaedBey}
\author{Mark van Hoeij}
\address{Department of Mathematics\\ Florida State University\\ Tallahassee, FL, USA}
\email{hb20@fsu.edu\\ hoeij@math.fsu.edu}

\keywords{difference operators, solving in terms of lower order, algorithms, absolute factorization.}

\begin{document}

\begin{abstract}
Our paper \cite{bou2024solving} gave two algorithms for solving difference equations in terms of lower order equations:
an algorithm for absolute factorization, and an algorithm for solving third order equations in terms of second order.
Here we improve the efficiency for absolute factorization, and extend the other algorithm to order four.

\end{abstract}
\maketitle
\section{Introduction}

Computer algebra systems are often used to solve differential equations from many branches of science. Linear homogeneous differential equations with rational function coefficients are very common in mathematics, combinatorics, physics and engineering. There are numerous algorithms to find closed form solutions (solutions expressible in terms of well studied special functions, for example; Bessel, Kummer, Liouvillian, Hypergeometric functions etc.) \cite{van1997factorization, debeerst2008solving, van2010finding, fang20112}. One way to interpret a closed form solution is that it is an expression that reduces a differential equation to a standard equation (the equation for that special function).
Viewed this way, it also makes sense to reduce differential equations to any equations of lower order, regardless of whether the reduced equation belongs to a special function. The papers \cite{singer1985solving, van2007solving}
show how to reduce, whenever possible, a differential equation of order three to an equation of order two, while 
\cite{van2002decomposing, hessinger2001computing,person2002solving} handle differential equations of order four. Our goal is to do this for recurrence relations, which we represent using a difference operator.

The first algorithm to compute hypergeometric solutions of recurrence relations was given by Petkovsek \cite{petkovvsek1992hypergeometric}. The first algorithm for Liouvillian solutions of recurrence relations was given by Hendriks and Singer \cite{hendriks1998algorithm, hendricks1999solving}. Both can be interpreted as writing solutions of a difference operator in terms of solutions of operators of order one. Our goal is to extend this to solving difference operators in terms of solutions of operators of order two. We call these solutions \emph{2-expressible} \cite{kaedbey2025solving}. (In this terminology, Liouvillian solutions are 1-expressible). 

One of the reasons for doing this is because of solvers for order 2 such as~\cite{Cha}. Another reason is timing. Computing 1-expressible (i.e. Liouvillian) solutions reduces to computing order-1 factors of certain auxiliary operators. To go beyond that, one needs an implementation that can find higher order factors, which now exists (RightFactors in Maple's LREtools package).

Our ISSAC paper \cite{bou2024solving} gives algorithms to find 2-expressible solutions for recurrences of order~3. There were three cases to be handled (factorization, absolute factorization, and gauge-equivalence to a symmetric square). To prove that, we needed to develop \cite{kaedbey2025solving} theory about absolute factorization and prove results using difference Galois theory.

For a 2-solvable order~4 operator $L \in D := \mathbb{C}(x)[\tau] $, there are four cases:
\begin{enumerate}
    \item $L$ is reducible in $D$;
    \item $L$ is irreducible but not absolutely irreducible;
    \item $(D/DL)\pda^{1}_{2}$ is isomorphic to a tensor product of two 2-dimensional modules;
    \item $L$ is gauge equivalent to some $L_2^{ \cs 3} \cs (\tau-r),$ where $L_2$ has order 2.
\end{enumerate}

Finding 2-expressible solutions for recurrences of order 4 is the first main goal of this paper. To accomplish this goal we need an algorithm for each case. For Case~(1) see \cite{zhou2022algorithms, barkatou2024hypergeometric,Rfactorimp}
while
Case~(2), absolute factorization, is treated in \cite{bou2024solving} with efficiency improvement in section 3. Thus, the first main goal in this paper consists of developing algorithms for Cases (3) and (4).

For Case (3), we develop an algorithm, in Section 4.2, that reduces $L$ (or its section operator) via a gauge transformation to a symmetric product of two order two operators. It computes the exterior square of $L$ and checks for order 3 factors that can be reduced to order 2 with \cite{bou2024solving}.
For Case (4) we give an algorithm in Section~4.3 to find the data mentioned in Case~(4).
There are two cases, depending on whether the symmetric square of $L$ has order 7 or order 10.

The second goal in this paper consists of
efficiency improvements for absolute factorization from \cite{bou2024solving}, and similar efficiency improvements for the
algorithms presented in this paper. To achieve this goal, we give results about determinants of difference operators and modules in Section~2. These are then used in Section 3 to improve absolute factorization by reducing the number of combinations that the factorizer needs to check, using Theorem~\ref{TheoremEfficiency} to discard unnecessary combinations.
Subsection~3.1 explains in what situations this is beneficial (in what situations the number of cases will be high), and we include an example that shows the computational time and number of cases before and after the improvement. We also give determinant formulas for the order~4 algorithms to reduce the number of cases there as well.
Implementations and experiments are available at~\cite{algo}.

\section{Difference Operators and Modules, Symmetric and Exterior Powers and Determinants}

Let $D = \mathbb{C}(x)[\tau]$ be the noncommutative polynomial ring in $\tau$ over $\mathbb{C}(x)$ with multiplication rule $\tau\circ f(x)=f(x+1)\circ \tau,$ for all $f$ in $\mathbb{C}(x)$.
An element
\begin{equation} \label{eqL}
    L=a_n\tau^n+a_{n-1}\tau^{n-1}+\dots+a_{1}\tau+a_{0}		
\end{equation}
of $D$ is called a \emph{difference operator}.
Such $L$ acts on functions by $L(f)(x)=\sum_{i} a_i\,f(x+i).$
We assume that $a_0$ and $a_n$ in~(\ref{eqL}) are non-zero. Then the {\em order} of $L$ is n, denoted by $\ord(L)$.

A difference operator is said to be \emph{2-solvable} if it has a nonzero \emph{2-expressible} solution, which means a solution that can be
expressed\footnote{\label{note1}using difference ring operations, indefinite summation, and interlacing; see \cite{kaedbey2025solving} for more details}
in terms of solutions of second order equations.
Algorithms to find $1$-expressible solutions already exist~\cite{hendricks1999solving}, such solutions are called Liouvillian solutions.

Let $S=\mathbb{C}^{\mathbb{N}}/\sim$ where $u \sim v$ when $u-v$ has finite support.
This is a difference ring under component-wise addition and multiplication. 
It is also a $D$-module; $\mathbb{C}(x)$ embeds into $S$ by evaluation at $\mathbb{N}$ because rational functions have finitely many poles.

The \emph{solution space} $V(L)$ is the set $\{u\in S\mid L(u)=0\}$.
This is a $\mathbb{C}$-vector space of dimension $\ord(L)$ by \cite[Theorem 8.2.1]{petkovvsek1997wilf}.

\begin{definition} \label{Def21} (Cyclic vector, minimal operator, irreducible modules).
Apart from $S$ we only consider $D$-modules $M$ that are finite dimensional $\mathbb{C}(x)$-vector spaces.
An element $u\in M$ is a \emph{cyclic vector} if $Du=M$. The {\em minimal operator} of $u$, denoted $\Minop(u,D)$,
is the monic generator of the left ideal $\{L\in D \mid L(u)=0\} \subset D$.

A module $M$ \emph{irreducible} if every nonzero element of $M$ is cyclic; 
equivalently, the minimal operator of every nonzero element of $M$ is irreducible of order $\dim(M)$.
\end{definition}

\begin{lemma} 
\label{lemma:gauge}
    Let $L_1,L_2\in D$ have the same order. Then the following statements are equivalent:
    \begin{enumerate} 
        \item $D/DL_1\cong D/DL_2$ as $D$-modules.
        \item There exists $G \in D$ with $G(V(L_2))=V(L_1).$
        \item There exists $G \in D$ with $L_1G \in DL_2$ and $\GCRD(G,L_2)=1$. (The same $G$ as in (2)).
        \item There exists a $D$-module with cyclic vectors $u_1,u_2$ such that $L_1=\Minop(u_1,D)$ and $L_2=\Minop(u_2,D).$  
    \end{enumerate}  
\end{lemma}

\begin{definition}  \label{def:gauge} If any statement in Lemma \ref{lemma:gauge} holds, we say that $L_1$ is \emph{gauge equivalent} to $L_2$,
    and $G$ is a \emph{gauge transformation} from $L_{2}$ to $L_{1}$.
    Moreover,
    $G: V(L_2) \rightarrow V(L_1)$ is a bijection.  \end{definition}

\begin{definition} \label{DefProj}
Operators $L, L' \in D$ are \emph{projectively equivalent} if there exists $L_1 \in D$ of order 1 for which $L_1 \circledS L$ is gauge equivalent to $L'$.
A \emph{projective transformation} is a pair $(L_1,G),$ where $G$ is a gauge transformation from $L_1 \circledS L$ to $L'$. This expresses solutions of $L'$ in terms of solutions of $L_1$ and $L$. In particular, if $L$ is 2-solvable, so is $L'$.
\end{definition}

\begin{definition} (Matrix for $M$).
	If $M$ is a $D$-module then we say that \emph{$A$ is a matrix for $M$}
	if it {\em expresses the action} of $\tau$ on a basis of $M$.
	If $b_1,\ldots,b_n$ is that basis, then \[
		\left(\begin{matrix} \tau(b_1) \\ \vdots \\ \tau(b_n) \end{matrix}\right)
		= A \left(\begin{matrix} b_1 \\ \vdots \\ b_n \end{matrix}\right).
	\]
\end{definition}

\begin{remark} \label{ChBas} (Change of basis).
If $A$ is a matrix for $M$ then so is $\tau(P) A P^{-1}$ for any $P \in \GL_n(\mathbb{C}(x))$.
\end{remark}

\begin{definition} (Determinant).
	If $M$ is a $D$-module of dimension $n$, then $\det(M)$ is the 1-dimensional $D$-module $\bigwedge^n M$.
\end{definition}

We typically represent $\det(M)$ with a rational function as follows.

\begin{definition} \label{DefSim}
	If $M_1$ is a 1-dimensional $D$-module and $r\in \mathbb{C}(x)-\{0\}$ then we write $M_1 \sim r$ when
	$M_1 \cong D/D(\tau-r)$, i.e., when $(r)$ is a $1 \times 1$ matrix for $M_1$.
	If $r_1, r_2 \in \mathbb{C}(x)-\{0\}$
	then we write $r_1 \sim r_2$ when $D/D(\tau - r_1) \cong D/D(\tau - r_2)$ which by Remark~\ref{ChBas} is equivalent to
	$r_1/r_2 = \tau(P)/P$ for some $P \in \mathbb{C}(x)-\{0\}$.
\end{definition}

\begin{definition} \label{DefCompanion} (Companion matrix and determinant of $L$). The  \emph{companion matrix}
of $L$ from equation~(\ref{eqL}) is:
$$C_L := \begin{pmatrix}
   0      & 1      & 0      & \cdots & 0 \\
   0      & 0      & 1      & \cdots & 0 \\
   \vdots & \vdots & \vdots & \smash{\ddots} & \vdots \\
   0      & 0      & 0      & \smash{\ddots} & 1 \\
   -a'_0    & -a'_1    & -a'_2    & \cdots & -a'_{n-1} 
  \end{pmatrix}$$
where $a'_i = a_i / a_n$ (recall that $a_0$ and $a_n$ are non-zero).
It expresses the action of $\tau$ on the basis $1,\ldots,\tau^{n-1}$ of $D/DL$.
Because of Lemma~\ref{lemmaDet} below we define
the \emph{determinant} of $L$ as
$\det(L) := \det(C_L) = (-1)^n {a'_0}
= (-1)^n a_0/a_n$.
\end{definition}

\subsection{Symmetric and Exterior powers} \label{Section21}

Let $k$ be a field. Let $V$ be a $k$-vector space with basis $B := \{b_1,\ldots,b_n\}$.
By viewing $b_1,\ldots, b_n$ as variables, we obtain a polynomial ring $k[b_1,\ldots, b_n]$.
We identify $\Sym^d(V)$ with the set of homogeneous polynomials of degree $d$.
Let $B_d := \{b_1^d, b_1^{d-1} b_2, \ldots, b_n^d\}$ be a basis.

If $A$ is the matrix w.r.t. $B$ of a linear map $\phi: V \rightarrow V$,
then the {\em symmetric power matrix}\, $\Sym^d(A)$ is the matrix w.r.t. $B_d$ of the induced
linear map $\Sym^d(V) \rightarrow \Sym^d(V)$. For elementary matrices it is easy to check that
\begin{equation} \label{eqDetSym}
	\det(\Sym^d(A)) = \det(A)^m \ \ \ {\rm where} \ \ \ m = \binom{n+d-1}{n}. 
\end{equation}
Then~(\ref{eqDetSym}) holds in general because both $\det$ and $\Sym^d$ preserve products.
Likewise, $\phi$ induces a linear map $\bigwedge^d(V) \rightarrow \bigwedge^d(V)$. We denote its matrix w.r.t. basis
$\{ b_{i_1} \wedge \cdots \wedge b_{i_d} \ | \ i_1 < \cdots < i_d \}$ as $\bigwedge^d(A)$.
\begin{equation} \label{eqDetExt}
	\det(\bigwedge\nolimits^d(A)) = \det(A)^k \ \ \ {\rm where} \ \ \ k = \binom{n-1}{d-1}.
\end{equation}

If $M$ and $N$ are $D$-modules, then
the \emph{tensor product} $M\otimes N$ is a $D$-module under $\tau(m \otimes n) := \tau(m) \otimes \tau(n)$.
The $D$-module structure of symmetric and exterior powers of $M$ is defined in a similar way.

\begin{lemma} \label{lemmaDet} Let $M$ be an $n$-dimensional $D$-module and let $A$ be a matrix for $M$. Then
$\det(M) \sim \det(A)$,\,  $\det(\Sym^d(M)) \sim \det(A)^m$, and  $\det(\bigwedge^d(M)) \sim \det(A)^k$
with $m,k$ as in equations~(\ref{eqDetSym}) and~(\ref{eqDetExt}).
If $M \cong D/DL$ then $\det(M) \sim \det(C_L) = \det(L)$.
\end{lemma}

\begin{proof}
Let $b_1,\ldots,b_n$ be a basis of $M$ and $b := b_1 \wedge \ldots \wedge b_n$ a basis of $\det(M)$. Now $\tau(b) = \tau(b_1) \wedge \ldots \wedge \tau(b_n) = \det(A)b$,\,
so $\det(M) \sim \det(A)$.
The formulas for $\Sym^d(M)$, $\bigwedge^d(M)$ and $D/DL$ hold because $\Sym^d(A)$, $\bigwedge^d(A)$ and $C_L$ are matrices for these modules.
\end{proof}

\subsection{Symmetric and Exterior powers of operators} \label{Section22}

If $L \in D$ has order $n$ then $b_1,\ldots,b_n := \tau^0,\ldots,\tau^{n-1}$ is a basis of $M := D/DL$.
Let $L^{\cs d}$ be the minimal operator of $b_1^d \in \Sym^d(M)$.
If $b_1^d$ is a {cyclic vector} (if it generates $\Sym^d(M)$)
then we say the order of  $L^{\cs d}$ is {\em as expected}, in which case
\[ \ord( L^{\cs d} ) = \dim (\Sym^d(M)) = \binom{n+d-1}{d} \]
and
\[ \det(  L^{\cs d} )  \sim  \det( L)^{\binom{n+d-1}{n} } \] 
by Lemma~\ref{lemmaDet}.
If $b_1^d$ is not a cyclic vector, then we say that the order of $L^{\cs d}$ is {\em lower than expected}.

Likewise,  let $\bigwedge^d(L)$ be the minimal operator of $b_1 \wedge \ldots \wedge b_d \in \bigwedge^d(M)$. If this is a cyclic vector then we say that the order is as expected, in which case
\[ \ord( \bigwedge\nolimits^d(L) ) = \dim ( \bigwedge\nolimits^d(M) ) = \binom{n}{d} \]
and
\[ \det(  \bigwedge\nolimits^d(L) )  \sim  \det( L)^{ \binom{n-1}{d-1}}. \]

\begin{definition}
\label{def:symm-prod} Let $L,L' \in D$.
The \emph{symmetric product} $L \cs L'$ is the minimal operator of $1\otimes 1 \in (D/DL)\otimes(D/DL')$. 
\end{definition}

Note that $L^{\cs 2}  =  L\cs L$. If $L_1$ has order 1 and $\lambda$ is a non-zero solution of $L_1$ then
$V(L_1 \circledS L) = \lambda \cdot V(L)$.

\begin{lemma} \cite[Zehfuss determinant]{zehfuss1858ueber}
\label{det_tensor}
Let $M_1, M_2$ be $n$ and $m$-dimensional $D$-modules respectively and let $A, B$ be the matrices for $M_1, M_2$ respectively. Then, $$\det(A \otimes B) = \det(A)^m\det(B)^n.$$     
\end{lemma}

\section{Absolute Factorization} \label{SectionAbsFactor}
\begin{definition} (Restriction of a $D$-module  {\cite[Definition 5.1]{kaedbey2025solving}}).
    Let $D_{p}=\mathbb{C}(x)[\tau^p] \subseteq D$, where $p \in \mathbb{Z}^{+}$.
    Let $M$ be a $D$-module.
    Then $M\pda^{1}_{p}$ is $M$ viewed as a $D_{p}$-module.
\end{definition}

\begin{definition}
\label{def:section-operator} (Section operator).
If $L \in D$ then let $L\pda^{1}_{p}$ be the minimal operator of $1  \in M\pda^1_p$ where $M := D/DL$.
Equivalently, $L\pda^{1}_{p}$ is the monic element in $D_{p}\cap DL$ of minimal order.
The $p$'th \emph{section operator} is $L^{(p)}\coloneqq \psi_{p}^{-1}(L\pda^{1}_{p}) \in D$.
Here $\psi_{p}$ is the following isomorphism
    \begin{equation}\begin{split} \label{isom-D-Dm}
 \psi_{p} :D &\longrightarrow D_p \\
  \tau &\mapsto \tau^p \\
  x &\mapsto \frac{x}{p}.\end{split}\end{equation} 
\end{definition}

These $L\pda^{1}_{p}$ and $L^{(p)}$ are the same as $P(\phi^p)$ and $P_{0}(\phi)$ from~\cite[Lemma 5.3]{hendricks1999solving}.
A submodule of $M\pda^{1}_{p}$ is a subset of $M$ that is also a $D_p$-module.
Every $D$-module is a $D_p$-module, but not vice versa, so $M\pda^{1}_{p}$ could have more submodules than $M$. Because of this, although the image of $1$ in $M := D/DL$ is always a cyclic vector of $M$,
it need not be a cyclic vector of $M\pda^{1}_{p}$ because it could generate
a proper submodule. In that case $L\pda^{1}_{p}$
and $L^{(p)}$ have {\em lower order than expected} (in which case the algorithm below stops in Step~1(b)).

\begin{definition} (Absolute irreducibility). A $D$-module $M$ is \emph{absolutely irreducible} if $M\pda^{1}_{p}$ is irreducible for every $p \in \mathbb{Z}^{+}$.
We say that $L$ is absolutely irreducible when $D/DL$ is.
\end{definition}

\noindent \textbf{Algorithm:} \texttt{AbsFactorization} from  \cite{bou2024solving}. \\
\textbf{Input:} An irreducible operator $L\in D$.\\
\textbf{Output:} \texttt{Absolutely\,Irreducible} if $L$ is absolutely irreducible; otherwise $[p,\{R_i\}]$. Here $p$ is a prime
and each $R_i$ right-divides $L^{(p)}$, with $\psi_{p}(R_i)$ generating a nontrivial submodule of $(D/DL) \pda^{1}_{p}$.
\begin{enumerate}
    \item For each prime factor $p$ of $n := \ord(L)$\ \begin{enumerate}
        \item Compute $L^{(p)}$ from Definition \ref{def:section-operator}.
        \item If $\ord(L^{(p)}) < \ord(L)$ then return $[p,\{1\}]$ and stop. \label{1b}
        \item Let $S\coloneqq \RightFactors\left(L^{(p)}, n/p \right)$. \label{1c}
        \item \label{1d} If $S \neq \varnothing$, then return $[p,S]$ and stop.
    \end{enumerate}
    \item Return \texttt{Absolutely\,Irreducible}.
\end{enumerate}

The main step is to factor $L^{(p)}$ in $D$.
Section~\ref{Section31} below discusses factoring and gives an example to illustrate an efficiency issue, which will be addressed in Section~\ref{SectionImproved}.

\subsection{Factorization in $D$} \label{Section31}
Manuel Bronstein showed that finding $d$'th order right-factors of $L$ can be done by computing
hypergeometric solutions of a system  $\tau(Y) = AY$ where $A = \bigwedge^d(C_L)$. He also made significant progress towards solving such systems,
see \cite{barkatou2024hypergeometric} for more details.

{If} we know the so-called {\em types} of hypergeometric solutions for a system $\tau(Y) = AY$, then finding hypergeometric solutions reduces to computing {\em polynomial solutions} of related systems $\tau(P) = \lambda_i^{-1} A P$, see \cite[Sections~4 and~5]{barkatou2024hypergeometric} for details.
Thus, a key goal in \cite{barkatou2024hypergeometric} is to construct a set of {\em candidate types} (denoted $\mathcal{H}_2 \subset \mathbb{C}(x)$ in \cite{barkatou2024hypergeometric}). This set should be as small as possible, because for each $\lambda_i \in \mathcal{H}_2$
we have to compute\footnote{The CPU time depends on the degree bound for $P$, computed for each $\lambda_i$ in (\cite[Section~5]{barkatou2024hypergeometric} and \cite[Section 4.5]{zhou2022algorithms}). However, we have no reasonable a priori degree bound.} the polynomial solutions of $\tau(P) = \lambda_i^{-1} A P$. 

We will call $S$ a set of {\em candidate determinants} for $L$ if $\det(R)$ provably appears in $S$ up to $\sim$
for every order-$d$ right-factor $R \in D$ of $L$.
Our goal is to minimize $S$ because it equals the set of candidate types for $\tau(Y) = AY$ when $A = \bigwedge^d(C_L)$.

\begin{example} \label{ExampleA}
The cost of Step~1(c) in algorithm \texttt{AbsFactorization} depends
on the number of candidates in the factorization, which in turn depends on the number of
factors of the leading and trailing coefficients.
To illustrate this we constructed an example of a fourth order difference operator where $L^{(2)}$ is reducible and has many small factors in the leading and trailing coefficients: \\
$L = a_4\tau^4 + a_3\tau^3 + \cdots +a_0\tau^0,$ where {\tiny \sloppy $a_4=(x+4)^{2}(x+5)^{2}(2x+7)^{2}(2x+9)^{2}(4x-11)(524160x^{8}+9391200x^{7}-118179432x^{6}-253541284x^{5}-339259113x^{4}-283416626x^{3}-140532705x^{2}-35130024x-2220048),$ $a_3=16(x+4)^{2}(2x+7)^{2}(524160x^{12}+15113280x^{11}-364158816x^{10}-4278491572x^{9}-9186978746x^{8}+12166953346x^{7}-86741410290x^{6}-843333775440x^{5}-2144077451746x^{4}-3001904754612x^{3}-2506144851117x^{2}-1178353117620x-242095406175),$ $a_2= (-137438945280x^{17}-6482503372800x^{16}-90355220358144x^{15}-154953056569984x^{14}+8139627355615616x^{13}+69179680108818000x^{12}+277321791062784832x^{11}+698868352509149328x^{10}+1236863662787672992x^{9}+1625448731323698944x^{8}+1626145247262854144x^{7}+1235819925815197696x^{6}\\+686291085150978048x^{5}+244593652122419200x^{4}+24045290042818560x^{3}-27607241721839616x^{2}-15602879836717056x-2930851407200256),$ $a_1=32768(x+2)^{2}(2x+3)^{2}  (1048320x^{12}+38613120x^{11}-475672512x^{10}-11499544808x^{9}-68147233556x^{8}-184773020492x^{7}-262836346620x^{6}-216526023556x^{5}\\
-122659285853x^{4}-39783078178x^{3}+1029344695x^{2}+7429526904x+2484513522),$ and $a_0=4096(x+1)^{2}(x+2)^{2}(2x+3)^{2}(2x+1)^{2}(4x+33)(524160x^{8}+13584480x^{7}-37764552x^{6}-736049716x^{5}-3014273813x^{4}-6181409598x^{3}-7122549901x^{2}-4430333096x-1162363872)$.}

First we check if $L$ is reducible by computing
factors of orders $d=1,2,3$.
Setting \verb +infolevel[LREtools]:=10+ causes \texttt{RightFactors} (\cite{Rfactorimp} and \cite[Section 7]{barkatou2024hypergeometric}) to display every candidate it checks, plus some other information.
For $d=2$ it checks 30 candidates in order to find all hypergeometric solutions of the 6 by 6 system $\bigwedge^d(C_L)$.
This means it computes polynomial solutions of 30 systems. Each of these is converted to a system over~$\mathbb{Q}$
using degree bounds computed from generalized exponents \cite[Chapter 6]{zhou2022algorithms}.
After that, the number of unknowns in $\mathbb{Q}$ of these systems ranges from 13 to 43, and the number of equations ranges from 74 to 117.
Most of the 30 systems have no solutions, and the remaining ones have solutions that do not correspond to factors.
All of this takes only a fraction of a second and we find that $L$ is irreducible in $D$.

Factoring $L^{(2)}$ in Step 1(c) takes much longer, primarily because the
number of candidate types is much higher for $L^{(2)}$ than it is for $L$ (the order is still 4 but the degree in $x$ is much higher). This time \texttt{RightFactors} reports 1791 candidate combinations (see Remark~\ref{RemCombination} below). Degree bounds and the resulting systems over $\mathbb{Q}$ are larger too.
\end{example}

\begin{remark} \label{RemCombination} One technical point in the implementation \texttt{RightFactors}  is that it constructs
more than the candidate types explained in~\cite{zhou2022algorithms, barkatou2024hypergeometric}.
It constructs {\em candidate combinations} for the factors $R$ to be computed.
A combination for $R$ contains not only the type (i.e.  $\det(R)$ up to $\sim$) but also the Newton polygon and polynomials of $R$.
(For a definition see \cite{ChaThesis}).
Due to a relation between these ingredients, the number of candidate combinations could be less than the number of candidate types,
but it could also be greater, because one candidate type could match more than one Newton polygon plus polynomials.
But the general situation remains that the CPU time for {RightFactors}
depends greatly on the number of candidates. The purpose of Theorem~\ref{TheoremEfficiency} below is to discard candidates
when possible.
\end{remark}

\subsection{Efficiency Improvement of Absolute Factorization} \label{SectionImproved}

\begin{theorem} \label{TheoremEfficiency}
Let $L \in D$ be irreducible of order $n$, and let $p$ be a prime dividing $n$.
Suppose that $L^{(p)}$ has order $n$ as well.
Then $\psi_p( \det(R)) \sim \det(L)$ for any factor $R$ of $L^{(p)}$ of order $n/p$.
\end{theorem}

By passing this information to RightFactors in Step 1(c) on algorithm \texttt{AbsFactorization},
the number of candidate combinations in Example~\ref{ExampleA} reduces from 1791 to 363,
resulting in a significant improvement in running time, from 24.7 to 6.9 seconds.
The reason the number of cases is not $1$ is because $\psi_p(\lambda_1) \sim \psi_p(\lambda_2)$ does not imply $\lambda_1 \sim \lambda_2$.

Our algorithm is also an efficient way to compute Liouvillian solutions, which corresponds to order~1 factors of $L^{(n)}$. \\

\noindent {\bf Proof of Theorem~\ref{TheoremEfficiency}:}
When $L$ is irreducible and $L^{(p)}$ is not, we have $D/DL \cong D/D\tilde{L}$ for some $\tilde{L} \in D_p$, see \cite[Theorem 4.1]{bou2024solving}.
Now $D/D\tilde{L}$ is an irreducible $D$-module with basis $B = \{\tau^0,\ldots,\tau^{n-1}\}$. Let $M_i := {\rm SPAN}_{\mathbb{C}(x)}(B_i) \subset D/D\tilde{L}$
where $B_i := \{ \tau^j \in B \ | \ j \equiv i \mod p\}$. 
This $M_i$ is a $D_p$-module because $\tilde{L}$ is in $D_p$. Now 
$\tilde{L}_i := \tilde{L}\vert_{x \mapsto x+i} = \tau^i \tilde{L} \tau^{-i} \in D_p$ annihilates $\tau^i \in M_i \subset D/D\tilde{L}$ and hence $M_i \cong D_p / D_p \tilde{L}_i$. {($\tau^i \tilde{L} \tau^{-i}$ is computed in the larger ring $\mathbb{C}(x)[\tau,\tau^{-1}] \supseteq D$)}.
We find
\[ (D/DL)\pda^{1}_{p} \ \cong \ (D/D\tilde{L})\pda^{1}_{p}  \ = \ \bigoplus_{i=0}^{p-1} M_i  \ \cong \  \bigoplus_{i=0}^{p-1} \left(D_p/D_p\tilde{L}_i\right). \]

If $N$ is a $D_p$-module then we can use the isomorphism $\psi_p: D \rightarrow D_p$ to construct a $D$-module
$[N] := \{ [n] \ | \ n \in N \}$
by defining $G \cdot [n] := [\psi_p(G) \cdot n]$ for any $G \in D.$

The minimal operator of $[1] \in [ (D/DL)\pda^{1}_{p} ]$ is $L^{(p)}$. We assumed $\ord(L^{(p)}) = n$, which implies $D/DL^{(p)} \cong [ (D/DL)\pda^{1}_{p} ]$,
because otherwise the algorithm stops in Step 1(b).
So
\[ D/DL^{(p)} \, \cong \,  [(D/DL)\pda^{1}_{p} ] \ \cong \  \bigoplus_{i=0}^{p-1} [D_p/D_p\tilde{L}_i] \ \cong \   \bigoplus_{i=0}^{p-1} \left( D/DL_i \right)  \]
where $L_i := \psi_p^{-1}(\tilde{L}_i)$ is irreducible in $D$ (because $\tilde{L}$ and $\tilde{L}_i$ are irreducible in $D_p$, and even irreducible in $D$).
This shows that if $R$ is any factor of $L^{(p)}$ of order $n/p$, then $D/DR \cong D/D L_i$ and $\det(R) \sim \det(L_i)$ for some $i$.
Then $\psi_p( \det(R)) \sim \det( \tilde{L}_i) \sim  \det( \tilde{L} )  \sim \det(L)$.

\begin{remark} (continuation of Remark~\ref{RemCombination})
As mentioned in Remark~\ref{RemCombination}, our factoring implementation constructs not only candidates for $\det(R)$ up to $\sim$,
but it combines them with candidates for the Newton polygon and polynomials for $R$. 
But we know the latter data ahead of time.
This is because gauge equivalent operators $L$ and $\tilde{L}$
have the same Newton polygon and polynomials, and $\psi_p^{-1}(\tilde{L})$ gives this data for $R$.
Combinations with the wrong Newton polygon or polynomials can thus be discarded,
which further reduces the number of cases from 363 to 121, and the CPU time from 6.9 to 3.2 seconds.
\end{remark}

In total we have a reduction from 24.7 to 3.2 seconds
for the absolute factorization of $L$ from Example~\ref{ExampleA}.
The improvement coming from Theorem~\ref{TheoremEfficiency} will be smaller for examples where $a_0,a_n$ have fewer factors, because if there are fewer combinations then there is less room for improvement.
See \cite{algo} for more examples.

\section{Solving equations of order 3 and 4 in terms of order 2}

\subsection{Lists of cases}
If $L \in D$ have order 3, then \cite{kaedbey2025solving} proved that $L$ is $2$-solvable when one of the following cases holds:
\begin{enumerate}
\item Reducible case: $L$ is reducible in $D$.
\item Liouvillian case: $L$ is gauge equivalent (Definition~\ref{def:gauge}) to $\tau^{3}+a$ for some $a\in \mathbb{C}(x)$. 
\item Symmetric square case: $L$ is gauge equivalent to $L_{2}^{\cs 2}\cs L_{1}$ for some $L_1, L_2 \in D$ of orders 1 and 2.
\end{enumerate}
Algorithms for these cases are given in \cite[Section~7]{barkatou2024hypergeometric}, \cite[Section 3.1]{bou2024solving}, and \cite[Section 5.2]{bou2024solving} respectively.
\begin{remark} \label{RemOnly}

The condition in case~(3) above only determines $L_2$ up to projective equivalence. {(The proof is similar to the differential case, \cite[Theorem 4.7]{van2006descent}).} So $L_2$, computed with algorithm ReduceOrder~\cite{bou2024solving}, is not unique, it is only unique up to projective equivalence.
In fact, the size of this $L_2$ can vary significantly. To improve this, our website \cite{algo} contains code to find, given $L_2$ as input, a projectively equivalent operator of near-optimal size.
\end{remark}
For $L$ of order 4 we have the following 2-solvable cases:
\begin{enumerate}
    \item $L$ is reducible (if a right-factor has order~3 then apply~\cite{bou2024solving}).
    \item $L$ is irreducible but not absolutely irreducible.
    \item $L^{(2)}$ is gauge equivalent to $L_{2a} \cs L_{2b}$ for some $L_{2a}, L_{2b} \in D$ of order 2. 
    \item $L$ is projective equivalent to $L_2^{ \cs 3} \cs L_1$ for some $L_1,L_2 \in D$ of orders 1 and 2.
\end{enumerate}

The proof that this list for order~4 is complete will be in the first author's PhD thesis, using tools from \cite{kaedbey2025solving}. See~\cite{barkatou2024hypergeometric} for the algorithm of case~(1)
and see Section~\ref{SectionAbsFactor} for the algorithm of case~(2).
Here the goal is to give algorithms to find 2-expressible solutions for the remaining cases $(3)$ and $(4)$
with improvements similar to Section~\ref{SectionImproved}.

\subsection{Case (3): symmetric product}

If an irreducible operator $L$ of order 4 is gauge equivalent to a symmetric product of second order operators,
then $\bigwedge^2(L)$ has two right-factors of order three, both 2-solvable {(see \Cref{proof:}, or \cite[Section 4.4]{person2002solving} for the differential case)}. 

We divide Case (3) into two sub-cases: Case 3(a) is when $L$ itself is gauge equivalent to a symmetric product, and Case 3(b) is when $L^{(2)}$, but not $L$, is gauge equivalent to a symmetric product. Theorem~\ref{TheoremEfficiency} provides a significant efficiency improvement for Case~3(b).
The implementation can be found at \cite{algo}. \\

\noindent \textbf{Algorithm:} \texttt{Case 3(a)} \\
\textbf{Input:} An irreducible order 4 operator $L$.\\
\textbf{Output:} If the exterior square of $L$ does not have order 3 right factors then return FAIL, otherwise return two order 2 operators $L_{2a} \text{ and } L_{2b}$ and
a gauge transformation $G: V(L_{2a}\circledS L_{2b}) \to V(L)$ (see Remark~\ref{RemG} and \Cref{def:gauge} for the definition of gauge transformation).

\begin{enumerate}
    \item Let $L_6 := \bigwedge^2(L)$
    (see Section~\ref{Section22}).

    \item Let $S = \RightFactors(L_6,3)$. If $S = \emptyset$ then return FAIL. \\
    Otherwise $S$ should have two elements, say $L_{3a}$ and $L_{3b}$.

    \item Apply the ReduceOrder algorithm \cite{bou2024solving} to $L_{3a}$ and $L_{3b}$ to obtain $L_{2a}$ and $L_{2b}$ of order~2. If this fails then return FAIL.

    \item $L_s:=L_{2a} \cs L_{2b}$.

    \item \label{step5} Find an operator $G$ that maps $V(L_s)$ to $V(L)$, see Remark~\ref{RemG}.

    \item Return $L_{2a},L_{2b}$ and $G$.
\end{enumerate} 

\begin{remark} \label{RemG} 
As mentioned in Remark~\ref{RemOnly}, the operators $L_{2a},L_{2b}$ are only unique up to projective equivalence. So we are free to replace say $L_{2a}$ by $L_{2a} \cs (\tau - r)$
for any non-zero $r$, which would replace $L_s$ by $L_s \cs (\tau - r)$. This would change $\det(L_s)$ by a factor $r^4$.
This $r$ needs to be chosen in the right way, because a necessary condition
for a gauge transformation $G$ from $L_s$ to $L$ to exist is that $\det(L_s) \sim \det(L)$.
This condition determines $r$ up to $\sim$ and up to a factor in $\{1, i, -1, -i\}$. For each of these 4 cases (2 cases if we restrict to working over $\mathbb{Q}$) we have to compute a gauge transformation, e.g. with \cite{Imp}. Another option is to use an implementation for projective equivalence called
\texttt{ProjectiveHom}, which finds both $r$ and the gauge transformation \cite{algo}.
\end{remark}

\noindent \textbf{Algorithm:} \texttt{Case 3(b)} \\
\textbf{Input:} An irreducible order 4 operator $L$ that is not gauge equivalent to a symmetric square. \\
\textbf{Output:} Same as {Case 3(a)} except with $L$ replaced with $L^{(2)}$. \\[5pt]

\noindent {\bf Inefficient version:}
\begin{enumerate}
\item Simply apply algorithm Case 3(a) to $L^{(2)}$.
\end{enumerate}
{\bf Efficient version:}   
\begin{enumerate}
\item Let $L_6 := \bigwedge^2(L)$.
\item Use the efficiency improvement in Section~\ref{SectionImproved} to compute order~3 right-factors $L_{3a},L_{3b}$ of $L_6^{(2)}$ (FAIL if not found).
\item Apply ReduceOrder to $L_{3a}$ to find $L_{2a}$  (FAIL if not found).\\
Substitute $x \mapsto x + 1/2$ in $L_{2a}$ to find $L_{2b}$. \\
(The proof of Theorem~\ref{TheoremEfficiency} shows that
this substitution swaps $L_{3a}$ and $L_{3b}$ up to gauge-equivalence).
\item Now proceed as in Case 3(a).
\end{enumerate}

\begin{example} (OEIS A227845) (also used in \cite{kaedbey2025solving}).
Let $$a(n) = \sum_{k=0}^{ [n/2] } \sum_{j=k}^{n-k} \binom{n-k}{j}^2 \binom{j}{k}^2$$
with $L = (x + 4)^2\tau^4 - 2(3x^2 + 21x + 37)\tau^3 + 2(3x^2 + 15x + 19)\tau - (x + 2)^2$.

With algorithm Case~3(b) we found that
$L^{(2)}$ is not just gauge equivalent, but actually equal to $L_{2a} \cs L_{2b}$ for some $L_{2a}$ and $L_{2b}$ where $L_{2b}$ is a $x \mapsto x + 1/2$ shift from $L_{2a}$.
We can write $a(2n)$ as a product of their solutions.

As for $a(n)$ itself, we listed the formula $a(n) = U(n) \cdot U(n-1)$ on the OEIS, where $U(-1),U(1),U(3),\ldots$ and $U(0),U(2), U(4), \ldots$ satisfy the recurrence
$n^2 U(n) = 2(3n^2-3n+1)U(n-2) - (n-1)^2U(n-4)$
obtained by applying the isomorphism $\psi_2: D \rightarrow D_2$ to $L_{2b}$.
\end{example}

The examples here give only summaries of the computations (omitting large expressions such as gauge transformations). Detailed computations are provided online \cite[ReduceOrder.mw]{algo}.

\begin{example} (OEIS A247365) Let

$L =(16x^6 + 96x^5 + 237x^4 + 307x^3 + 222x^2 + 88x + 15)\tau^4 + (x + 1)(64x^8 + 800x^7 + 4244x^6 + 12430x^5 + 21920x^4 + 23837x^3 + 15726x^2 + 5872x + 972)\tau^3 +  (-256x^{10} - 3840x^9 - 25472x^8 - 98304x^7 - 244271x^6 - 408233x^5 -  464965x^4 - 357285x^3 - 178432x^2 - 53022x - 7290)\tau^2 + (-64x^9 - 800x^8 -  4244x^7 - 12422x^6 - 21868x^5 - 23719x^4 - 15610x^3 - 5847x^2 - 1010x - 6)\tau + 16x^6 + 192x^5 + 957x^4 + 2535x^3 + 3765x^2 + 2977x + 981$ .

Algorithm Case~3(a) finds two operators, projectively equivalent to $\tau^2 + \tau x - 1$ and $(2x - 1)\tau^2 + 2(4x^2 + 1)x\tau + 2x + 1$, which can be solved in terms
of Bessel sequences using the algorithm from \cite{Cha}.  This led to the following formula
$$a(n) = \left[2 \cdot \BesselI(n-1,2) \cdot \BesselK(2n-1,2)\right]$$
where $[$ \ $]$ rounds to the nearest integer.
\end{example}

\subsection{Case 4: Symmetric cube}
If an irreducible fourth order operator $L$ is projective equivalent to a symmetric cube, 
then like in~\cite{person2002solving}, either $L^{\cs 2}$ has order~7, or it has order~10 with factors of orders 3 and 7. 
The implementation is at \cite{algo}. \\

\noindent \textbf{Algorithm:} \texttt{Case 4} \\
\textbf{Input:} An irreducible order 4 operator $L$.\\
\textbf{Output:} FAIL if $L$ is not projectively equivalent to $L_2^{\cs 3}$ for some $L_2$ of order 2, else $L_2$ and a projective transformation. 

\begin{enumerate}
    \item Find $L^{\cs 2}$. 

    \item If $\ord(L^{\cs 2})=7$, use algorithm \texttt{SpecialCase} below.

    \item Let $L_3 = \RightFactors(L^{\cs 2}, 3)$. If no right factor was found then return FAIL.
    
    \item Apply ReduceOrder to $L_3$ obtain an order 2 operator $L_2$.

    \item $L_s:=L_2^{\cs 3}$.

    \item Find a projective transformation like in Remark~\ref{RemG}.

\end{enumerate}

The factorization in Step~3 in algorithm \texttt{Case~4} typically involves a huge number of candidate types, which makes
the algorithm very slow. A large speedup is obtained
if we tell $\RightFactors$
which candidates types to discard with the following lemma. We implemented~\cite{algo} an additional speedup by discarding candidate Newton polygon/polynomials that are incompatible with projective equivalence to a symmetric square.

\begin{lemma}

Any candidate type
in the factorization in Step~3 in algorithm \texttt{Case~4}
whose square is not $\sim \det(L)^3$ can be dicarded.     
\end{lemma}

\begin{proof}
Given $L$, the algorithm decides if $D/DL \cong \Sym^3(M_2) \otimes M_1$ for some modules $M_1$ and $M_2$ of dimension 1 and 2.
Let $d_1 := \det(M_1)$, $d_2 := \det(M_2)$, and $M := \Sym^3(M_2) \otimes M_1$. Then $\det(L) \sim \det(M) = d_1^4 d_2^6$, see \Cref{det_tensor}.
So if $\det(L)$ is not $\sim$ to a square, then the algorithm can give up.
Now $M_{10} := \Sym^2(M)$ ($\cong D/DL^{\cs 2}$ if this has order 10) has a surjective homomorphism to
$M_7 := \Sym^6(M_2) \otimes \Sym^2(M_1)$.

So its kernel, the 3-dimensional module $M_3$ that \texttt{RightFactors} should
find, has $\det(M_3) \sim \det(M_{10})/\det(M_7) = \det(M)^5 / (d_1^{14} d_2^{21}) = d_1^6 d_2^9$. The square of this is $\sim \det(L)^3$.  
\end{proof}

\noindent \textbf{Algorithm:} \texttt{SpecialCase}\\
\textbf{Input:} An irreducible order 4 operator for which $L_4^{\cs 2}$ has order 7. \\
\textbf{Output:} Same as algorithm \texttt{Case 4}. \\
\textbf{Explanation:}  Let $b_1,\ldots,b_4 := \tau^0,\ldots,\tau^3$ be the standard basis of $M := D/DL$,
then $B_2 := \{b_1^2,b_1 b_2,\ldots, b_4^2\}$ from Section~\ref{Section21} is a basis
of $M_{10} := \Sym^2(M)$, and  $L^{\cs 2}$ from Section~\ref{Section22} is the minimal operator of $b_1^2$.
If this has order~7 then $D/DL^{\cs 2} \cong M_7$ with $M_7$ as above. To obtain $L_3$
we have to compute a cyclic vector of $M_{10}/M_7$.  The computation of $L^{\cs 2}$ already involved computing a basis of $M_7$.
\begin{enumerate}
\item Pick any $b \in B_2$ with $b \not\in M_7$.
\item Compute $b, \tau(b), \tau^2(b), \tau^3(b) \in M_{10}$.
\item Solve linear equations to find $L_3$ with $L_3(b) \equiv 0$ mod $M_7$.
\item Proceed as in algorithm \texttt{Case 4}.
\end{enumerate}

\begin{example} (OEIS A219670). Let

$  L =  (x + 3)^2(x + 4)^2(x + 5)(2x + 3)(7x^4 + 56x^3 + 166x^2 + 216x + 105)\tau^4  - (x + 3)(x + 4)(2x + 3)(2x + 7)(70x^6 + 1050x^5 + 6406x^4 + 20337x^3 + 35449x^2 + 32244x + 12048)\tau^3  - 3(x + 3)(2x + 5)(490x^8 + 9800x^7 + 84910x^6 + 416150x^5 +  1261159x^4 + 2417840x^3 + 2860095x^2 + 1905600x + 546588)\tau^2 + 27(x + 2)^2(2x + 3)(2x + 7)(70x^6 + 1050x^5 + 6406x^4 + 20283x^3 + 35044x^2 + 31221x + 11178)\tau + 729(x + 1)^3(x + 2)^2(2x + 7)(7x^4 + 84x^3 + 376x^2 + 744x + 550).$

Our algorithm finds that $L$ is projectively equivalent to the symmetric cube of $\tau^2+(2x + 1)\tau - 3x^2$, based on the worksheet in \cite{algo}. After solving this in terms of OEIS sequences with \cite{Giles} and some simplification, we obtained the following formula
$$a(n) = A002426(n)^2A005717(n).$$
(Note:  A005717 is gauge equivalent to A002426;  $L$ is projective equivalent but not equal to a symmetric cube).
\end{example}

For more details on these examples, such as timings, the number of combinations the factorizer checks (before or after the efficiency improvements in this paper), or other examples, see \cite{algo}.

\appendix

\section{{Exterior Square of a Tensor Product}}
\label{proof:}

Let $M$, $N$ be $D$-modules with bases $\{m_1,m_2\}$ and $\{n_1,n_2\}$.
Then \(M\otimes N\) has this basis
\[
w_1=m_1\!\otimes\!n_1,\quad
w_2=m_1\!\otimes\!n_2,\quad
w_3=m_2\!\otimes\!n_1,\quad
w_4=m_2\!\otimes\!n_2
\]
and \(\wedge^2(M \otimes N)\) has basis $e_1,\ldots,e_6=$
\[w_1 \wedge w_2, \ w_1 \wedge w_3, \ w_1 \wedge w_4, \ w_2 \wedge w_3, \ w_2 \wedge w_4, \ w_3 \wedge w_4\]

Let $d_m = m_1 \wedge m_2$ resp. $d_n = n_1 \wedge n_2$ be a basis of 
$\wedge^2 M$ resp. $\wedge^2 N$.

\bigskip
\noindent
A classical decomposition gives
\[
{\wedge^2(M \otimes N) \;\cong\;
   (\operatorname{Sym}^2M \otimes \wedge^2N)
   \;\oplus\;
   (\wedge^2M \otimes \operatorname{Sym}^2N)
}.
\]
The isomorphism is given by
\[ e_1 \mapsto m_1^2 \otimes d_n, \quad (e_3-e_4)/2 \mapsto m_1m_2 \otimes d_n,  \quad e_6 \mapsto  m_2^2 \otimes d_n\]
\[ \mbox{} \ e_2 \mapsto d_m \otimes n_1^2, \quad  (e_3+e_4)/2 \mapsto d_m \otimes n_1n_2, \quad \ e_5 \mapsto  d_m \otimes n_2^2.\]
To verify this, observe that $\tau$ acts on
$\{e_1,(e_3-e_4)/2,e_6,e_2,(e_3+e_4)/2,e_5\}$ as it acts on $\{m_1^2 \otimes d_n,\ldots,d_m \otimes n_2^2\}$, this action is given by
\[
\begin{pmatrix}
\det(B) \rho_3(A) & 0_{3\times3} \\
0_{3\times3} & \det(A) \rho_3(B)
\end{pmatrix}.
\]
Here $A,B$ express the action of $\tau$ on $\{m_1,m_2\}$ and $\{n_1,n_2\}$, while $\rho_3(A),\rho_3(B)$ express the action on $\{m_1^2, m_1m_2, m_2^2\}$ and $\{n_1^2, n_1n_2, n_2^2\}$.
The $3 \times 3$ matrix $\rho_3(A)$ is also found in the proof of 
\cite[Lemma 3.2]{singer1985solving}.

\bibliographystyle{plain}
 \newcommand{\noop}[1]{}

\end{document}